\documentclass[10pt]{amsart}
\usepackage{amssymb,latexsym}
\usepackage[mathscr]{eucal}

\theoremstyle{plain}
\newtheorem{theorem}{Theorem}
\newtheorem{lemma}{Lemma}
\newcommand{\refT}[1]{Theorem~\ref{T:#1}}
\newcommand{\refL}[1]{Lemma~\ref{L:#1}}
\newcommand{\refS}[1]{Section~\ref{S:#1}}
\numberwithin{equation}{section}

\newcommand{\be}[1]{\begin{equation}\label{E:#1}}
\newcommand{\refE}[1]{\eqref{E:#1}}

\newcommand{\q}{\quad}

\newcommand{\lp}{\left(}
\newcommand{\rp}{\right)}

\newcommand{\lf}{\left\lfloor}
\newcommand{\rf}{\right\rfloor}
\newcommand{\lc}{\left\lceil}
\newcommand{\rc}{\right\rceil}

\newcommand{\Blp}{\Bigl(}
\newcommand{\Brp}{\Bigr)}

\newcommand{\ff}{\frac}

\newcommand{\ga}{\alpha}
\newcommand{\gb}{\beta}
\newcommand{\gd}{\delta}
\newcommand{\gre}{\epsilon}

\newcommand{\gf}{\varphi}
\newcommand{\gt}{\tau}
\newcommand{\gw}{\omega}

\newcommand{\tcp}{\Phi_t}
\newcommand{\qr}{Q_\rho}
\newcommand{\qt}{Q_\gt}
\newcommand{\ca}{\mathcal{A}}
\newcommand{\at}{\ca_\gt}
\newcommand{\rep}{\mathcal{R}_{p,q}}
\newcommand{\tr}{\{p,q,r\}}
\newcommand{\ts}{\{p,q,s\}}

\begin{document}

\title[Inclusion-Exclusion Polynomials]
{On Ternary Inclusion-Exclusion Polynomials}
\author{Gennady Bachman}
\address{University of Nevada, Las Vegas\\
 Department of Mathematical Sciences\\
 4505 Maryland Parkway\\
 Las Vegas, Nevada 89154-4020}
\email{bachman@unlv.nevada.edu}
\subjclass{11B83, 11C08}
\keywords{Cylcotomic polynomials, inclusion-exclusion polynomials}

\begin{abstract}
Taking a combinatorial point of view on cyclotomic polynomials
leads to a larger class of polynomials we shall call the
inclusion-exclusion polynomials. This gives a more appropriate
setting for certain types of questions about the coefficients
of these polynomials. After establishing some basic
properties of inclusion-exclusion polynomials we turn to a
detailed study of the structure of ternary inclusion-exclusion
polynomials. The latter subclass is exemplified by cyclotomic
polynomials $\Phi_{pqr}$, where $p<q<r$ are odd primes. Our main
result is that the set of coefficients of $\Phi_{pqr}$ is simply
a string of consecutive integers which depends only on the residue
class of $r$ modulo $pq$.
\end{abstract}
\maketitle

\section{Introduction}\label{S:1}

A ternary cyclotomic polynomial  is a cyclotomic polynomial $\tcp$
where $t$ is a product of three distinct odd primes. More precisely,
\[
\tcp(z)=\prod_{\substack{0<a<t\\(a,t)=1}}\lp z-e^{2\pi ia/t}\rp,
\]
where $t=pqr$ and $p$, $q$, and  $r$ are distinct odd primes.
Following the usual conventions we assume that $p$ is the smallest
of the three primes, let $a_m=a_m(t)$ denote the coefficients of
$\tcp$, and set
\[
A(t)=\max_m|a_m(t)|.
\]
There has been much progress recently in our understanding of
coefficients of $\tcp$ and, especially, of the function $A(t)$.
A long-standing conjecture of M. Beiter \cite{Be} asserted that
the bound
\be{1.1}
A(t)=A(pqr)\le\ff{p+1}2
\end{equation}
holds for all $t$. But in a recent work of Y. Gallot and P. Moree
\cite{GM} this conjecture was disproved in a rather dramatic
fashion and a number of prescriptions of integers $t$ for which
\refE{1.1} fails to hold were given. In particular, it was shown
that if $\gre>0$ is fixed then for every sufficiently large prime
$p$ there exist $q$ and $r$ such that
\be{1.D}
A(t)>\Blp\ff23-\gre\Brp p.
\end{equation}
They conjectured that, in fact,
\be{1.A}
A(t)\le\ff23p.
\end{equation}
Following this B. Lawrence \cite{La} announced that he proved the
validity of \refE{1.A} for $p>10^6$.

Another interesting question that was resolved recently is whether
it is possible to arrange it so that
\be{1.2}
A(t)=1,
\end{equation}
even for arbitrary large $p$. We say that polynomial $\tcp$ is flat
in this case. An old folklore conjecture asserts that there are
flat cyclotomic polynomials of all orders -- the order of $\Phi_n$
is the number of distinct odd prime divisors of $n$ if $n$ is not
a power of 2 and is 1 otherwise. The case of ternary
cyclotomic polynomials has now been settled in the affirmative and
the validity of \refE{1.2} was first established by
this author in \cite{B3}. This result was later extended by
T. Flanagan \cite{Fl} who showed that \refE{1.2} holds for a
larger family of integers $t$. But the best known result in
this direction is due to N. Kaplan \cite{Ka} who showed that
\refE{1.2} holds for every $t$ with $r\equiv\pm1\pmod{pq}$.

As part of his work on \refE{1.2} Kaplan showed that the value of
$A(t)$ is completely determined by the residue class of $r$ modulo
$pq$, where $r>q>p$. More precisely, he showed that if $s>q$ is
another prime and if $r\equiv\pm s\pmod{pq}$ then
\be{1.B}
A(pqr)=A(pqs).
\end{equation}
Moreover, he also obtained the following partial analogue of
\refE{1.B} for the set of coefficients $\ca_t=\{a_m(t)\}$ of $\tcp$.
Set identities
\be{1.C}
\ca_{pqr}=\begin{cases}
\ca_{pqs},  &\text{if $r\equiv s\pmod{pq}$,}\\
-\ca_{pqs}, &\text{if $r\equiv -s\pmod{pq}$,}
\end{cases}\end{equation}
are certainly valid if $r,s>pq$. The first of these identities
was also proved by Flanagan \cite{Fl}. Actually \refE{1.C} is only
implicit in \cite{Ka}. It follows that each residue class
$r_0$ modulo $pq$ determines at most two different sets of
coefficients $\ca_{pqr}$ with $r\equiv r_0\pmod{pq}$.

Residue class of $r$ modulo $pq$ imposes certain structure on
$\tcp$ and its set of coefficients $\ca_t$ and our object here is
to further investigate this structure. The main result of this
paper is that $\ca_t$ is completely determined by the residue class
of $r$ modulo $pq$, that is, that \refE{1.C} holds for $r,s>q$.
We are also interested in an analogue of \refE{1.C} for the case
when $r\equiv s\pmod{pq}$ but $s<q<r$. Kaplan's result on
$r\equiv\pm1\pmod{pq}$ falls into this case and is seen to be a
special case of this general principle (see \refS{3}). In pursuing
this development we shall work in a more general setting of what
we shall call inclusion-exclusion polynomials. Accordingly we begin
with a brief discussion of this class of polynomials and their
relation to cyclotomic polynomials, this is the subject of
\refS{2}. We then concentrate on the ternary case (of the
inclusion-exclusion polynomials) in \refS{3}.

\section{Inclusion-Exclusion Polynomials}\label{S:2}

Let $\rho=\{\,r_1,r_2,\dots,r_s\,\}$ be a set of natural numbers
satisfying $r_i>1$ and $(r_i,r_j)=1$ for $i\ne j$, and put
\[
n_0=\prod_ir_i,\q n_i=\ff{n_0}{r_i},\q
n_{ij}=\ff{n_0}{r_ir_j}\ [i\ne j],\q \dots
\]
For each such $\rho$ we define a function $\qr$ by
\be{2.1}
\qr(z)=\ff{\lp z^{n_0}-1\rp\cdot\prod_{i<j}
\lp z^{n_{ij}}-1\rp\cdot\dots}
{\prod_i\lp z^{n_i}-1\rp\cdot
\prod_{i<j<k}\lp z^{n_{ijk}}-1\rp\cdot\dots} .
\end{equation}
Our first observation is that $\qr$ is, in fact, a polynomial.

\begin{theorem}\label{T:1}
We have
\be{2.A}
\qr(z)=\prod_\gw(z-\gw),
\end{equation}
where the product is taken over all the roots of unity $\gw$
satisfying the condition
\[
\gw^{n_0}=1\q\text{but}\q\gw^{n_i}\ne1\q[1\le i\le s].
\]
Moreover, the degree of $\qr$ is given by
\[
\gf(\rho)=\prod_i(r_i-1).
\]
\end{theorem}

\begin{proof}
The claim follows by routine applications of the
inclusion-exclusion principle and we omit the straight-forward
details.
\end{proof}

We shall refer to polynomials $\qr$ as inclusion-exclusion
polynomials, the term suggested by the construction \refE{2.1}.
Our interest in this class of polynomials is motivated by the study
of coefficients of cyclotomic polynomials, as will become plain
below. From the algebraic point of view the most interesting case
is when parameters $r_i$ are assumed to be distinct prime numbers.
In this case $n_0=\prod_ir_i$ is a canonical factorization of $n_0$
into primes, $n_0$ is a square free integer, and the product in
\refE{2.A} is taken over all primitive $n_0$th roots of unity $\gw$.
In other words, in this case polynomial $\qr$ is better known as
cyclotomic polynomial $\Phi_{n_0}$. In general, as we show presently,
$\qr$ is a certain product of cyclotomic polynomials.

\begin{theorem}\label{T:2}
Given $\rho$ let
\be{2.2}
D=D_{\rho}=
\{\, d:d\mid n_0\text{ and }(d,r_i)>1\text{ for all }i\, \}.
\end{equation}
Then we have
\be{2.3}
\qr(z)=\prod_{d\in D}\Phi_d(z).
\end{equation}
\end{theorem}

\begin{proof}
Since both sides of \refE{2.3} are monic polynomials with roots
of multiplicity 1 it suffices to show that the roots are, in
fact, the same. So consider any root $\gw$ of $\qr$ and let $d$
be the smallest integer such that $\gw^d=1$. Then $\gw$ is a
root of $\Phi_d$. Moreover, by \refT{1}, $d\in D$ since
$d\mid n_0$ and $(d,r_i)>1$.

In the opposite direction, fix $d\in D$ and let $\gw$ be any
root of $\Phi_d$. Then $\gw^{n_0}=1$. Moreover, $(d,n_i)<d$ and
we conclude that $\gw^{n_i}\ne1$. Thus, by \refT{1}, $\gw$
is a root of $\qr$.
\end{proof}

We now turn our attention to the question of coefficients of these
polynomials. We begin with a few remarks of general nature on
coefficients of inclusion-exclusion polynomials versus coefficients
of cyclotomic polynomials. As is well known, to study coefficients
of cyclotomic polynomials it suffices to consider only polynomials
$\Phi_n$ with $n$ square free. We have seen that in this case
$\Phi_n=\qr$, where $r_i$ are prime factors of $n$. It is thus
natural to consider properties of coefficients of cyclotomic
polynomials in the larger context of coefficients of
inclusion-exclusion polynomials. In fact, the literature on this
topic contains many results on cyclotomic polynomials that are
actually theorems about inclusion-exclusion polynomials. This
characterization certainly applies to every result which was
obtained by an argument that (i) used identity \refE{2.1} as a
point of departure and (ii) did not crucially depend on parameters
$r_i$ to be prime but actually only required the condition
$(r_i,r_j)=1$, for $i\ne j$. For instance, most of the work on
coefficients of cyclotomic polynomials of low order falls into this
category. Even when it comes to open questions about coefficients of
cyclotomic polynomials, it seems rather clear that some of them are
truly questions about coefficients of inclusion-exclusion
polynomials. Take, for instance, the question of whether there
exist flat cyclotomic polynomials of arbitrary large order. It is
quite evident that this is a question about the structure of
\refE{2.1} and that the algebraic distinction of cyclotomic
polynomials has no bearing on the matter.

Finally, we remark that the setting of inclusion-exclusion
polynomials may offer not only a more appropriate context but it
may actually furnish certain ``practical'' advantages over the
setting of cyclotomic polynomials. This will become plain in the
next section where we take up the case of ternary
inclusion-exclusion polynomials (defined below).

We define the order of $\qr$ to be $s$, if $s=1$ or if $r_i\ge3$
for $1\le i\le s$, and $s-1$ otherwise. This parameter corresponds
to the order of cyclotomic polynomial and plays an important role.
The situation is identical to the more familiar setting of
cyclotomic polynomials and the best known examples of this are
polynomials of orders 1 and 2. Indeed, if $s=1$ and
$\rho=\{\, p\,\}$ ($p\ge2$; not necessarily prime), then
\be{2.B}
Q_{\{p\}}(z)=\ff{z^p-1}{z-1}=\sum_{n=0}^{p-1}z^n.
\end{equation}
Similarly,
\begin{align*}
Q_{\{p,q\}}(z)&=\ff{(z^{pq}-1)(z-1)}{(z^q-1)(z^p-1)}=
(1-z^{pq})(1-z)\sum_{i=0}^\infty z^{iq}\sum_{j=0}^\infty z^{jp}\\
&\equiv(1-z)\sum_{i,j\ge0}z^{iq+jp}\pmod{z^{(p-1)(q-1)+1}},
\end{align*}
by \refT{1}. It follows that if $\chi$ denotes the characteristic
function of integers representable in the form $iq+jp$ with
$i,j\ge0$ then
\be{2.C}
Q_{\{p,q\}}(z)=\sum_{n=0}^{(p-1)(q-1)}\chi(n)z^n-
\sum_{n=0}^{(p-1)(q-1)-1}\chi(n)z^{n+1}.
\end{equation}
Thus, in the sense of \refE{2.B} and \refE{2.C}, the structure of
$\qr$ is determined by the order. In particular, polynomials of
orders 1 and 2 are flat. (For a more detailed discussion of
polynomials of order 2, phrased in terms of cyclotomic polynomials,
see, for example, \cite{Le}.)

The condition $r_i\ge3$ in the definition of order of $\qr$ is
explained by the identity (whose cyclotomic polynomials analogue
is also well-known)
\[
Q_{\{2,r_2,\dots,r_s\}}(z)=Q_{\{r_2,\dots,r_s\}}(-z)\q[s\ge2].
\]
This follows readily from \refT{1} and we omit the details (see,
for example, \cite{Le}). This takes us to polynomials of order
$\ge3$ where the situation is considerably more interesting. As we
already mentioned in the introduction, even the ternary case, that
is $\qr$ of order 3, still presents interesting challenges. This
case is the principal object  of this paper and it will be taken
up in the next section.

The fact that cyclotomic polynomials are reciprocal proved to
be useful in the study of their coefficients. We conclude this
section by observing that the same is true for
inclusion-exclusion polynomials. Indeed, the identity
\[
\qr(z)=z^{\gf(\rho)}\qr(z^{-1})
\]
follows readily from \refE{2.1} and the fact that $\gf(\rho)$
is the degree of $\qr$. From this we infer that if
\[
\qr(z)=\sum_{m=0}^{\gf(\rho)}a_mz^m \q[a_m=a_m(\rho)],
\]
then $a_m=a_{\gf(\rho)-m}$.

\section{The Ternary Case}\label{S:3}

We shall write $\qt$ to denote a ternary inclusion-exclusion
polynomial. For esthetic reasons we normally write
$\gt=\{\, p,q,r\,\}$ rather than $\gt=\{\, r_1,r_2,r_3\,\}$. Thus,
contrary to the conventions of \refS{1}, we now assume only that
parameters $p$, $q$, and $r$ are $\ge3$ and relatively prime in
pairs. At times, however, the use of notation
$\gt=\{\, r_1,r_2,r_3\,\}$ will prove to be the better choice.
Consequently, we consider the two forms to be interchangeable and
shall freely use either one with our choice dictated by convenience.
Adopting other conventions in the introduction we write, by \refT{1},
\be{3.1}
\qt(z)=\sum_{m=0}^{\gf(\gt)}a_mz^m\q[a_m=a_m(\gt)],
\end{equation}
as well as
\be{3.A}
\at=\{\, a_m(\gt)\,\}\q\text{and}\q A(\gt)=\max_m|a_m(\gt)|.
\end{equation}
Moreover, set
\[
A^+(\gt)=\max_ma_m(\gt)\q\text{and}\q A^-(\gt)=\min_ma_m(\gt).
\]
Let us emphasize that we are not assuming any particular order for
the parameters $p$, $q$, and $r$. The structural symmetry of $\qt$
with respect to these parameters is a key aspect of the problem and
it will play an important role in our development. Correspondingly,
we shall explicitly state any additional assumptions on $p$, $q$,
and $r$ when it is appropriate.

Recall from the introduction that we are after the relationship
between polynomials $Q_{\tr}$ and $Q_{\ts}$ with
$r\equiv s\pmod{pq}$. This problem splits into two parts according
to whether
\[
r,\, s>\max(p,q)\q\text{or}\q r>\max(p,q)>s\ge1,
\]
say. The principal focus of this paper is the former condition
and our main result is as follows.

\begin{theorem}\label{T:3}
Set of coefficients $\ca_{\tr}$ is a string of consecutive integers
and, for $r>\max(p,q)$, is completely determined by the residue
class of $r$ modulo $pq$. More precisely, we have
\be{3.B}
\at=[\, A^-(\gt),A^+(\gt)\,]\cap\mathbb Z
\end{equation}
and, for $r,s>\max(p,q)$,
\be{3.C}
\ca_{\tr}=\begin{cases} \ca_{\ts}, &\text{if $r\equiv s\pmod{pq}$,}\\
 -\ca_{\ts}, &\text{if $r\equiv -s\pmod{pq}$.}
 \end{cases}\end{equation}
\end{theorem}

As we pointed out in \refS{2}, much of what is known about
cyclotomic polynomials of low order pertains to corresponding
inclusion-exclusion polynomials. The work of Flanagan and Kaplan
discussed in the introduction is a case in point. In particular,
identity \refE{3.C} was already known to hold under the assumption
$r,s>pq$.

The remark in the preceding paragraph applies to all of our
references in what follows. Thus, for the sake of simplicity, we
shall henceforth ignore the distinction between cyclotomic and
inclusion-exclusion polynomials, when appropriate.

We deduce \refE{3.B} from \refL{2} below. Both of these facts were
discovered independently by Gallot and Moree \cite{GM2}. It is
worth noting that the approach in the works of \cite{Ka} and
\cite{GM2} is rather different from ours.

We derive \refT{3} by a sequence of lemmas some of which are
of independent interest and shed additional light on the
structure of $\qt$. As we remarked earlier, our development
preserves symmetry in the parameters $p$, $q$, and $r$ whenever
it is appropriate. This, in particular, will be handy when
considering the second alternative, namely
\be{3.D}
r\equiv\pm s\pmod{pq}\q\text{and}\q r>\max(p,q)>s\ge1.
\end{equation}

The situation in this case is more complicated and $A(p,q,r)$
need not equal to $A(p,q,s)$ -- in a slight abuse of notation
we shall write $A(p,q,r)$ in place of $A(\{\, p,q,r\,\})$, and
use the same conventions for the functions $A^+$ and $A^-$.
Instead we have the following result. Recall from \refS{2} that
polynomials of order less than 3 are flat, so that $A(p,q,2)=1$.
Moreover, it is convenient to extend the definition of $A$ by
setting $A(p,q,1)=0$. With these conventions we state our result
for \refE{3.D}.

\begin{theorem}\label{T:4}
If $r$ and $s$ satisfy \refE{3.D}, then
\be{3.E}
A(p,q,s)\le A(p,q,r)\le A(p,q,s)+1.
\end{equation}
\end{theorem}

The proof of \refT{4} will require further development and will be
carried out elsewhere. We shall limit ourselves here to a few brief
remarks. Note that, by \refT{3}, we have
\[
A(p,q,r)=A(p,q,pq\pm s),
\]
under \refE{3.D}. In this light \refE{3.E} is seen as a recursive
estimate. Of course, using an absolute upper bound for $A(p,q,s)$
on the right of \refE{3.E} yields the corresponding upper bound for
$A(p,q,r)$. For instance, a simple and convenient estimate
\be{3.F}
A(p,q,r)=A(p,q,pq\pm s)\le s,
\end{equation}
valid for all $s\ge1$, is obtained on combining \refE{3.E} with
the bound (see \cite{B1})
\[
A(p,q,s)\le s-\lc s/4\rc.
\]
Estimate \refE{3.F} sacrifices precision for convenience and is
certainly weaker than \refE{3.E} for $s\ge5$. On the other hand,
\refE{3.F} is sharp for $1\le s\le3$. We do not know if the
equation $A(p,q,pq+4)=4$ has any solutions. Note also that Kaplan's
result on flat cyclotomic polynomials corresponds to \refE{3.F}
with $s=1$.

When iteration of \refE{3.E} is possible it leads to a very rapid
reduction technique. For example, two applications of \refE{3.E}
give
\[
A(p,pq+1,p^2q+p+q)\le A(p,pq+1,q)+1\le2.
\]

Our first step is to observe that, by \refE{2.1}, polynomial $\qt$
has a representation
\be{3.2}\begin{aligned}
\qt(z)&=\ff{(1-z^{pqr})(1-z^r)(1-z^q)(1-z^p)}
{(1-z^{qr})(1-z^{rp})(1-z^{pq})(1-z)}\\
&\equiv(1-z^r)(1-z^q)(1+z+\dots+z^{p-1}) \\
&\phantom{(1-z^r)(1-z^q)}\times \sum_{i=0}^\infty z^{iqr}
\sum_{j=0}^\infty z^{jrp}\sum_{k=0}^\infty z^{kpq}
\pmod{z^{pqr}}.
\end{aligned}\end{equation}
Evidently, of key importance are integers $n$ of the form
\be{3.3}
n=iqr+jrp+kpq,\q i,j,k\ge0,
\end{equation}
and we let $\chi=\chi_{\gt}$ be the characteristic function of
such integers, that is,
\be{3.4}
\chi(n)=\chi_{\gt}(n)=\begin{cases}
1, &\text{if $n$ has representation \refE{3.3},}\\
0, &\text{otherwise.}
\end{cases}\end{equation}
Note that if $n<pqr$ then either representation \refE{3.3} is
not possible or it is unique. Therefore, by \refE{3.1}, \refE{3.2}
and \refE{3.4}, the identity
\be{3.5}
a_m=\sum_{m-p<n\le m}\bigl(\chi(n)-\chi(n-q)-\chi(n-r)
+\chi(n-q-r)\bigr)
\end{equation}
holds for all $m<pqr$. Let us clarify the meaning of this statement.
Recall that $\qt$ is a polynomial of degree $\gf(\gt)$. But in
\refE{3.5} we take $a_m=0$ for $m<0$ and for $\gf(\gt)<m<pqr$ and
then the identity  remains valid in the range $m<pqr$. We shall
find this extension useful for technical reasons.

In considering integers representable in the form \refE{3.3} it
is helpful to observe that every integer $n$ has a unique
representation in the form
\be{3.6}
n=x_nqr+y_nrp+z_npq+\gd_npqr,
\end{equation}
with $0\le x_n<p,\ 0\le y_n<q,\ 0\le z_n<r,$ and
$\gd_n\in\mathbb{Z}$. It follows that $n$ is representable in the
form \refE{3.3} if and only if $\gd_n\ge0$. But if $n<pqr$, as we
shall assume henceforth, then $\gd_n\le0$ and we obtain the
characterization
\be{3.7}
\chi(n)=1\q\text{if and only if}\q \gd_n=0 \q[n<pqr].
\end{equation}

We shall deduce \refE{3.B} from \refL{1} below. Recall our
convention of using $\{\, p,q,r\,\}$ and
$\{\, r_1,r_2,r_3\,\}$ interchangeably.

\begin{lemma}\label{L:1}
We have
\[
\bigl|\chi(n)-\sum_i\chi(n-r_i)+\sum_{i<j}\chi(n-r_i-r_j)
-\chi(n-r_1-r_2-r_3)\bigr|\le1.
\]
\end{lemma}

\begin{proof}
Let $\gt'=\{\,\pm r_1,\pm r_2,\pm r_3\,\}$ and for every
pair of $u,v\in\gt'$ with $|u|\ne|v|$ set
\be{3.G}
\psi_{uv}(n)=\chi(n)-\chi(n-u)-\chi(n-v)+\chi(n-u-v).
\end{equation}
Observe that if $w$ is another element of $\gt'$ and $|u|$, $|v|$,
and $|w|$ are all distinct then
\[
\bigl|\chi(n)-\sum_i\chi(n-r_i)+\sum_{i<j}\chi(n-r_i-r_j)
-\chi(n-{\textstyle\sum_ir_i})\bigr|= \bigl|\psi_{uv}(n')-\psi_{uv}(n'-w)\bigr|,
\]
where $n'=n+(u-|u|)/2+(v-|v|)/2+(w-|w|)/2$. Therefore to prove the
lemma it suffices to show that the inequality
\be{3.H}
\psi_{uv}(n)-\psi_{uv}(n-w)\le1
\end{equation}
holds for all $n$ (in an appropriate range depending on parameters
$u$, $v$, and $w$).

In \cite[Lemma 2]{B1} it was shown that
\be{l11}
|\psi_{uv}(n)|\le1.
\end{equation}
Actually, this estimate was given explicitly only for $\psi_{qr}$,
with $q,r>p$, but the argument applies for every choice of
$u$ and $v$ in $\gt'$. The rest of this proof is essentially an
extension of \cite[Proof of Lemma 2]{B1} and, in particular,
\refE{l11} will serve as a convenient reduction tool. In the first
place, by \refE{3.H} and \refE{l11}, it suffices to
show that there is no $n$ for which
\be{l13}
\psi_{uv}(n)=1\q\text{and}\q\psi_{uv}(n-w)=-1.
\end{equation}
To reach a contradiction let us assume that \refE{l13} holds for
some $n$. Noting that, by \refE{3.G}, we have
\[
\psi_{uv}(n)=\psi_{(-u)(-v)}(n-u-v),
\]
shows that, in addition to $\psi_{uv}(n)=1$, there is no loss in
generality in assuming that $\chi(n)=1$. Similarly, by \refE{3.G}
and symmetry, it follows that in addition to $\psi_{uv}(n-w)=-1$
we may assume that $\chi(n-w-u)=1$. It now follows from \refE{l11}
with $v$ replaced by $w$ that we may also assume that $\chi(n-w)=1$,
say. But then, since $\psi_{uv}(n-w)=-1$, we must also have
$\chi(n-w-v)=1$ and $\chi(n-w-u-v)=0$. Finally, in view of
$\psi_{uv}(n)=1$, we may further assume that $\chi(n-u)=0$, say.
To summarize, to show that \refE{l13} is not
possible it suffices to show that there is no $n$ for which
\[
\chi(n)=\chi(n-w)=\chi(n-w-u)=\chi(n-w-v)=1
\ \text{and}\ \chi(n-u)=\chi(n-w-u-v)=0.
\]
But this readily follows from \refE{3.6} and \refE{3.7} by
chasing the coefficients in representations \refE{3.6} of all
the relevant integers (see \cite[Proof of Lemma 2]{B1}).
\end{proof}

\begin{lemma}\label{L:2}
We have $\bigl|a_m(\gt)-a_{m-1}(\gt)\bigr|\le1$.
\end{lemma}

\begin{proof}
By \refE{3.5} and \refE{3.G},
$a_m-a_{m-1}=\psi_{qr}(m)-\psi_{qr}(m-p)$,
and the claim follows from \refL{1}.
\end{proof}

The following simple observation is rather useful. By \refE{3.6} and
\refE{3.7}, we have
\begin{gather}
\chi(n)=\chi(n-pq),\q\text{unless $\gd_n=z_n=0$},\label{E:3.12}\\
\chi(n)=1\ \text{and}\  \chi(n-pq)=0,\q
 \text{if $\gd_n=z_n=0$},\label{E:3.13}
\end{gather}
as well as the analogues of \refE{3.12} and \refE{3.13} with $pq$
and $z_n$ replaced by $qr$ and $x_n$ or by $rp$ and $y_n$,
respectively.

\begin{lemma}\label{L:3}
Let $R_m$ be the set of all integers appearing as an argument of
$\chi$ in the summation \refE{3.5}. Then we have
\be{l3}
a_m=a_{m-pq},\q\text{unless there is $n\in R_m$ with $\gd_n=z_n=0$,}
\end{equation}
as well as the analogues of \refE{l3} with $pq$ and $z_n$ replaced
by $qr$ and $x_n$ or by $rp$ and $y_n$, respectively.
\end{lemma}

\begin{proof}
This is an immediate consequence of \refE{3.5} and \refE{3.12}.
\end{proof}

\begin{lemma}\label{L:3B}
The estimate
\be{3.I}
|a_m-a_{m-r_ir_j}|\le2
\end{equation}
holds unconditionally for every pair of parameters $r_i\ne r_j$.
Moreover, if $r\ge p+q$ then we have
\be{3.J}
|a_m-a_{m-pq}|\le1.
\end{equation}
\end{lemma}

\begin{proof}
Consider representation \refE{3.5} with parameter $p$ given by
$p=\min(r_1,r_2,r_3)$. We shall now prove \refE{3.I} with
$r_ir_j=qr$, the remaining case follows in the same way. Put
$I_n=(n-p,n]\cap\mathbb Z$, so that
\be{3.K}
R_m=I_m\cup I_{m-q}\cup I_{m-r}\cup I_{m-q-r}.
\end{equation}
Note that $x_n=0$ if and only if $n$ is a multiple of $p$. But
each of the four intervals on the right of \refE{3.K} contains
exactly one multiple of $p$, say $\ga_ip$. Therefore, by
\refE{3.5}, \refE{3.12} and \refE{3.13}, we get
\[
|a_m-a_{m-qr}|\le
|\chi(\ga_1p)-\chi(\ga_2p)-\chi(\ga_3p)+\chi(\ga_4p)|,
\]
and \refE{3.I} follows.

To prove \refE{3.J} we argue in the same way but take advantage
of the condition $r\ge p+q$. By \refL{3}, we may assume that $R_m$
contains multiples of $r$ -- these are integers $n$ with $z_n=0$.
But, in the present case, the range $I_m\cup I_{m-q}$ contains
at most one such multiple, say $\ga r$. Therefore,
by \refE{3.5}, \refE{3.12} and \refE{3.13}, we have
\[
|a_m-a_{m-pq}|=|\chi(\ga r)-\chi(\ga r-r)|\le1,
\]
as claimed.
\end{proof}

Lemmas 3 and 4 have a number of interesting consequences. Observe
that if we take $r$ to be the largest of the three parameters then,
by \refE{l3}, it suffices to consider coefficients $a_m$ with
$m=\ga r+\gb$ and $0\le\gb<\min(r,p+q)$. We will use this fact
below. The parallel between \refE{3.J} and \refL{2} is immediate.
Unlike \refL{2}, however, \refE{3.J} has a hole in the form of the
range $\max(p,q)<r<p+q$. Estimate \refE{3.I} may be used to give a
simple upper bound for coefficients of $\qt$ as follows. Let
$p<\min(q,r)$ and take $r_ir_j=qr$ in \refE{3.I}. One then readily
verifies that iterating \refE{3.I} yields the bound
\be{3.L}
|a_m|\le2\lc m/(qr)\rc+1.
\end{equation}
Recall that we may assume that $m\le\gf(\gt)/2<pqr/2$. Thus
\refE{3.L} suggests that coefficients of largest size are to
occur near the middle of the range of the index $m$ and gives a
nontrivial bound for ``small'' $m$. Note also that replacing 2 by
1 on the right of \refE{3.I} would have the same effect on the
right of \refE{3.L}. But this would imply the bound
$A(\gt)\le\lc p/2\rc+1$, contradicting \refE{1.D}. It follows
that, in general, \refE{3.I} is sharp.

Another immediate consequence of \refE{3.6} is that $\gd_n\ge0$
if and only if $x_nqr+y_nrp\le n$. Therefore, if
\be{3.8}
f(n)=f_{\gt,r}(n)=x_nq+y_np,
\end{equation}
then, by \refE{3.7},
\be{3.9}
\chi(n)=1\q\text{if and only if}\q f(n)\le\lf n/r\rf.
\end{equation}
This observation, first made in \cite{B2}, plays a key role in
our analysis. Note that $f(n)\equiv nr^*\pmod{pq}$, where $r^*$ is
the multiplicative inverse of $r$ modulo $pq$. Now let $[N]_{pq}$
denote the least nonnegative residue of $N$ modulo $pq$ and let
$\rep$ be the set of integers representable as a nonnegative linear
combination of $p$ and $q$, that is,
\[
\rep=\{\, N\mid N=xq+yp,\ x,y\ge0\,\}.
\]
Then, by \refE{3.6} and \refE{3.8}, we have
\be{3.10}
f(n)=\begin{cases} [nr^*]_{pq}, &\text{if $[nr^*]_{pq}\in\rep$}, \\
 [nr^*]_{pq}+pq, &\text{otherwise.}
\end{cases}\end{equation}
It is now evident that $f$ is determined by the residue class of
$r$ modulo $pq$. Before stating this formally, let us introduce the
convention of writing $f_r$ in place of $f_{\gt,r}$, as long as it
is understood that the parameters $p$ and $q$ are fixed. Similarly,
we will find it convenient to write $\chi_r$ in place of
$\chi_{\gt}$ under the same circumstances.

\begin{lemma}\label{L:4}
If $r\equiv s\pmod{pq}$ and $n_1\equiv n_2\pmod{pq}$ then
$f_r(n_1)=f_s(n_2)$.
\end{lemma}

\begin{proof}
Since $[n_1r^*]_{pq}=[n_2s^*]_{pq}$ the conclusion follows from
\refE{3.10}.
\end{proof}

In addition to assumptions of \refL{4} we need to impose certain
further restrictions in order to guarantee that
$\chi_r(n_1)=\chi_s(n_2)$.

\begin{lemma}\label{L:5}
Suppose that $\max(p,q)<r<s$ and that $r\equiv s\pmod{pq}$. Then
\be{l51}
\chi_r(kr+j)=\chi_s(ks+j),
\end{equation}
for all $k<pq$ and $|j|<r$. Moreover, if $|j|<\min(r,pq)$ then
we also have
\be{l52}
\chi_r(kr+j-r)=\chi_s(ks+j-r).
\end{equation}
\end{lemma}

\begin{proof}
The first claim follows from \refE{3.9} and \refL{4}. The second
claim with $j\ge0$ follows in exactly the same way (\refE{l51}
contains \refE{l52} for $j>0$). In the remaining case we have
\be{l53}
\Big\lfloor\ff{kr+j-r}r\Big\rfloor=k-2\q\text{and}\q
\Big\lfloor\ff{ks+j-r}s\Big\rfloor=k-1.
\end{equation}
Moreover, if $j\not\equiv0\pmod{pq}$ then
\be{l54}
f_r(kr+j-r)\ne k-1,
\end{equation}
since $f_r(kr+j-r)\equiv k-1+jr^*\pmod{pq}$. Combining
\refE{l53} and \refE{l54} with \refE{3.9} and \refL{4}
completes the proof of the lemma.
\end{proof}

We are now ready to consider functions $A^+(\gt)$ and $A^-(\gt)$.
In view of \refE{3.B}, these functions capture all the information
about the coefficients of $\qt$ as a set. It is plain from our
introductory discussion of the function $A(\gt)$ (in the form of
$A(t)$) that there are basic gaps in our understanding of these
functions. It is interesting to note that in contrast to this the
quantity $A^+(\gt)-A^-(\gt)$ is more transparent. Indeed,
it is known \cite{B1, B2} that the bound
\[
A^+(\gt)-A^-(\gt)\le p
\]
is valid for all $\gt$ and that it is sharp. Our present aim is
the identity \refE{3.C} for which we need to show that if $p$ and
$q$ are fixed then for $r>\max(p,q)$ functions $A^{\pm}(p,q,r)$
depend only on the residue class of $r$ modulo $pq$.

\begin{lemma}\label{L:6}
If $r,s>\max(p,q)$ and $r\equiv s\pmod{pq}$ then
\[
A^+(p,q,r)=A^+(p,q,s)\q\text{and}\q A^-(p,q,r)=A^-(p,q,s).
\]
\end{lemma}

\begin{proof}
Consider the function $A^+$. With $p$ and $q$ fixed, let us write
$A^+(r)$ for $A^+(\gt)$. It suffices to show that
if $r>\max(p,q)$ and $s=r+pq$ then
\be{l61}
A^+(r)=A^+(s).
\end{equation}

Throughout this argument we adopt the convention that $a_m$ and
$b_l$ denote coefficients of $Q_{\tr}$ and $Q_{\ts}$,
respectively. We show first that if $0\le\gb<r$ then
\be{l62}
a_{\ga r+\gb}=b_{\ga s+\gb}.
\end{equation}
Set $m=\ga r+\gb$ and $l=\ga s+\gb$. Observe that, by \refL{5},
we have
\be{l63}
\chi_r(m-j)=\chi_s(l-j),
\end{equation}
for all $0\le j<p+q$. Moreover, \refE{l63} also holds with $m$
and $l$ replaced by $m-r$ and $l-s$, respectively. Combining
this with \refE{3.5} establishes \refE{l62}.

Of course, the inequality $A^+(r)\le A^+(s)$ is an immediate
consequence of \refE{l62}. In fact, \refE{l62} implies \refE{l61}
for we will show that for some $\ga$ and $0\le\gb<r$ we have
\be{l64}
A^+(s)=b_{\ga s+\gb}.
\end{equation}
Thus it only remains to prove \refE{l64}.

Let $l_0$ be the smallest index for which $b_{l_0}=A^+(s)$.
Note that, by \refE{3.5}, $b_0=1$ and $b_q=-1$, so $l_0$ (as
well as the corresponding quantity for the function $A^-$) are
well defined -- see remarks following \refE{3.5}. Applying
\refE{l3} to $b_{l_0}$ shows that we must have
$l_0=\ga_0s+\gb_0$, with $0\le\gb_0<q+p$, and that $\ga_0s$
and $\ga_0s-s$ are in $R_{l_0}$ (given by \refE{3.K} with $r$
replaced by $s$). If, in fact, $\gb_0<r$ then we are done. So
suppose that $r\le\gb_0<q+p$. We claim that in this case \refE{l64}
holds with $\ga s+\gb=l_0+pq$. In the first place, we have
\[
l_0+pq=(\ga_0+1)s+\gb_0-r,\q 0\le\gb_0-r<r.
\]
Therefore, to prove \refE{l64} it remains to show that
\be{l65}
b_{l_0+pq}=b_{l_0}.
\end{equation}

Recall that $\ga_0s, \ga_0s-s\in R_{l_0}$ and observe that they
are the only multiples of $s$ in $R_{l_0}$. Moreover our
assumption on $\gb_0$ implies that
\begin{gather}
\ga_0s\in I_{l_0-q},\q\ga_0s-s\in I_{l_0-q-s},  \label{E:3.M} \\
\ga_0s+r\in I_{l_0},\q\text{and}\q\ga_0s-s+r\in I_{l_0-s}. \label{E:3.N}
\end{gather}
Therefore, by \refE{3.5}, \refE{3.12}, \refE{3.13}, and \refE{3.M},
we get
\be{l66}
b_{l_0-pq}=b_{l_0}-\chi_s(\ga_0s-s)+\chi_s(\ga_0s).
\end{equation}
Whence $\chi_s(\ga_0s)=0$ (and $\chi_s(\ga_0s-s)=1$, but we
will not need this). Also, by \refE{3.N}, $\ga_0s+s\in I_{l_0+pq}$
and $\ga_0s\in I_{l_0+pq-s}$, and they are the only multiples of
$s$ in $R_{l_0+pq}$. Therefore, reasoning as in \refE{l66},
we now get
\[\begin{aligned}
b_{l_0}&=b_{l_0+pq}+\chi_s(\ga_0s)-\chi_s(\ga_0s+s) \\
&=b_{l_0+pq}-\chi_s(\ga_0s+s).
\end{aligned}\]
This implies \refE{l65} and completes the proof in the case
of the function $A^+$.

Essentially identical argument works for the function $A^-$
and we omit the details.
\end{proof}

\begin{proof}[Proof of \refT{3}]
\refE{3.B} follows from \refL{2}.

The first conclusion in \refE{3.C} follows from \refL{6} and
\refE{3.B}.

Recall that Kaplan \cite{Ka} has proved \refE{3.C} for $r,s>pq$.
Using this we deduce the second conclusion in \refE{3.C} from
the first.

Since our method is different from that of \cite{Ka} it is of
interest to give a self-contained treatment for the case
$r\equiv-s\pmod{pq}$. We thus conclude this paper with a sketch of
our argument. We will show that if $a_m$ is an arbitrary coefficient
of $Q_{\tr}$ then there is a coefficient $b_l$ of $Q_{\ts}$
such that $b_l=-a_m$. To do this we make an additional assumption
that $r,s\ge p+q$; this is permissible in view of what we already
proved. By \refE{l3}, we may take $m=\ga r+\gb_1$, with
$0\le\gb_1<p+q$. Following Kaplan, we claim that
\[
a_{\ga r+\gb_1}=-b_{\ga s+\gb_2},\q\text{with}\ \gb_2=p+q+1-\gb_1.
\]
To this end we observe that $[(kr+j)r^*]_{pq}=[(ks-j)s^*]_{pq}$,
so that
\[
f_r(kr+j)=f_s(ks-j).
\]
This is the present case equivalent of \refL{4}. From this we
deduce that if $|j|<p+q$ then
\be{3.O}
\chi_r(kr+j)=\chi_s(ks-j),
\end{equation}
the equivalent of \refE{l51} in \refL{5}. Now apply \refE{3.O} to
the representations of $a_{\ga r+\gb_1}$ and $b_{\ga s+\gb_2}$
given by \refE{3.5}. The proof is completed on observing that the
left and the right sides of \refE{3.O} contribute with the opposite
signs to the values of $a_{\ga r+\gb_1}$ and $b_{\ga s+\gb_2}$,
respectively.
\end{proof}

\end{document}